\documentclass{amsart}

\usepackage{amssymb}
\usepackage{mathrsfs}

\usepackage{hyperref}
\usepackage{enumitem}

\usepackage{booktabs,caption}
\usepackage{longtable}
\usepackage{enumitem}

\usepackage{comment}
\usepackage{graphicx}
\usepackage{color}
\usepackage[all]{xy}

\usepackage[initials]{amsrefs}

\newtheorem{theorem}{Theorem}[section]

\newtheorem{prop}[theorem]{Proposition}
\newtheorem{question}[theorem]{Question}

\theoremstyle{definition}
\newtheorem{remark}[theorem]{Remark}

\newcommand{\stack}[2]{ \mathcal{M}^\mathrm{Kss}_{#1, #2} } 
\newcommand{\modspace}[2]{M^\mathrm{Kps}_{#1, #2}} 

\DeclareMathOperator\DefqG{Def^{\mathrm{qG}}}
\DeclareMathOperator\Aut{Aut}

\newcommand{\epsi}{\varepsilon}
\newcommand{\rmd}{\mathrm{d}}

\newcommand{\Xprime}{{X'}}

\newcommand{\into}{\hookrightarrow} 

\DeclareMathOperator{\Hom}{Hom} 
\DeclareMathOperator{\Spec}{Spec} 
\DeclareMathOperator{\Pic}{Pic}
\DeclareMathOperator{\Spf}{Spf} 

\newcommand{\bmu}{\boldsymbol{\mu}}

\def\conv#1{\mathrm{conv} \left\{ #1  \right\} } 
\def\pow#1{[ \! [ #1 ] \! ] }

\newcommand\cM{\mathcal{M}}
\newcommand\scrM{\mathscr{M}}
\newcommand\cO{\mathcal{O}}
\newcommand\cX{\mathscr{X}}

\renewcommand\AA{\mathbb{A}}
\newcommand\CC{\mathbb{C}}

\newcommand\PP{\mathbb{P}}
\newcommand\QQ{\mathbb{Q}}
\newcommand\RR{\mathbb{R}}
\newcommand\ZZ{\mathbb{Z}}

\newcommand\rH{\mathrm{H}}

\title[A 1-dimensional component of K-moduli of del Pezzo surfaces]{A 1-dimensional component \\ of K-moduli of del Pezzo surfaces}

\author{Andrea Petracci}

\address{Dipartimento di Matematica, Universit\`a di Bologna, Piazza di Porta San Donato 5, 40126 Bologna, Italy}

\email{a.petracci@unibo.it}

\begin{document}
   \begin{abstract}
We explicitly construct a component of the K-moduli space of K-polystable del Pezzo surfaces which is a smooth rational curve.
\end{abstract}
   \maketitle
   
   \section{Introduction}
   
One of the most important and recent results in K-stability and in the theory of Fano varieties is the construction of K-moduli \cite{ABHLX, xu_minimizing, BLX, jiang_boundedness, blum_xu_uniqueness, xu_zhuang, properness_K_moduli, projectivity_K_moduli_final, lwx}.
It has been proved that, for every positive integer $n$ and every positive rational number $V$, $\QQ$-Gorenstein families of K-semistable Fano varieties over $\CC$ of dimension $n$ and anticanonical volume $V$ form an algebraic stack $\stack{n}{V}$ of finite type over $\CC$. Moreover, this stack admits a good moduli space $\modspace{n}{V}$, which is a projective scheme over $\CC$, and the set of closed points of $\modspace{n}{V}$ coincides with the set of K-polystable Fano varieties over $\CC$ of dimension $n$ and anticanonical volume $V$.
We refer the reader to \cite{xu_survey} for a survey on these topics.

The case of smoothable del Pezzo surfaces has been extensively studied 
\cite{odaka2016, odaka_compact, mabuchi_mukai}.
Moreover, K-moduli are understood for cubic $3$-folds  \cite{liu_xu_cubic},
cubic $4$-folds \cite{liu_4folds},
and for certain pairs $(S,C)$ where $S$ is a surface and $C$ is a curve on $S$ \cite{ADL19, ADL20}.
	
The goal of this note is to show how toric geometry and deformation theory can help understanding the geometry of explicit components of K-moduli.
Similar ideas were used
in \cite{ask_petracci}   to construct examples of reducible or non-reduced K-moduli of Fano $3$-folds (see also \cite{petracci_murphy, petracci_embedded}),
 in \cite{liu_petracci}   to study the K-stability of certain del Pezzo surfaces with Fano index $2$,
 and in \cite{jesus_cristiano} to study the dimension of K-moduli.
In this note we analyse a specific example of K-polystable toric del Pezzo surface and we prove the following:

\begin{theorem} \label{thm:main}
There exists a connected component of $\modspace{2}{\frac{22}{15}}$ which is isomorphic to $\PP^1$.
\end{theorem}

It is natural to wonder about the following:

\begin{question}
	Does there exist $V \in \QQ_{>0}$ such that a connected component of $\modspace{2}{V}$ is
	 a smooth curve of positive genus?
	\end{question}

\subsection*{Outline}
In \S\ref{sec:deformations_1/4(1,1)} we briefly recall the deformation theory of the surface singularity given by the cone over the rational normal curve of degree $4$.
In \S\ref{sec:surface_X} we introduce a Fano polygon $P$ and a K-polystable toric del Pezzo surface $X$, and we analyse its deformation theory; in particular, we show that the connected component of the K-moduli space of K-polystable del Pezzo surfaces that contains $X$ is smooth and $1$-dimensional.
In \S\ref{sec:scaffolding} we prove that $X$ is a hypersurface in a toric $3$-fold $Y$ and in \S\ref{sec:linear_system_of_X_in_Y} we prove that deforming $X$ inside the linear system $\vert \cO_Y(X) \vert$ on $Y$ gives the versal deformation of $X$.
This gives a non-constant morphism from an open subset of $\vert \cO_Y(X) \vert$ to the K-moduli space.
In \S\ref{sec:conclusion_proof} we conclude the proof of Theorem~\ref{thm:main}.
In \S\ref{sec:mirror_symmetry} we sketch what mirror symmetry says in this context.

\subsection*{Notation and conventions}
We work over an algebraically closed field of characteristic zero, which is denoted by $\CC$.
A \emph{Fano} variety is a normal projective variety over $\CC$ such that its anticanonical divisor is $\QQ$-Cartier and ample.
A \emph{del Pezzo} surface is a Fano variety of dimension $2$.
We assume that the reader is familiar with toric geometry \cite{cox}.
Every toric variety we consider is normal.

If $r, a_1, \dots, a_n$ are integers and $r \geq 1$, then the symbol $\frac{1}{r}(a_1,\dots,a_n)$ stands for the quotient of $\AA^n$ under the action of the cyclic group $\bmu_r$ defined by $\zeta \cdot (x_1, \dots, x_n) = (\zeta^{a_1} x_1, \dots, \zeta^{a_n} x_n)$ for every $\zeta \in \bmu_r$.
We use the same symbol to indicate the \'etale-equivalence class of the singularity of this quotient variety at the image of the origin of $\AA^n$.

\subsection*{Acknowledgements}
The author learnt most of the techniques and the ideas described in this note during countless conversations with Tom Coates, Alessio Corti, Al Kasprzyk and Thomas Prince over the years; it is a pleasure to thank them.

\section{Proof}

\subsection{Deformations of $\frac{1}{4}(1,1)$} \label{sec:deformations_1/4(1,1)}

The cyclic quotient singularity $\frac{1}{4}(1,1)$ is the affine cone over the $4$th Veronese embedding of $\PP^1$ into $\PP^4$.
The deformations of this singularity have been studied by Pinkham \cite[\S4]{pinkham}. Here we concentrate on the $\QQ$-Gorenstein deformations -- see \cite[\S2]{petracci_murphy} for a quick recap.

The singularity $\frac{1}{4}(1,1)$ has Gorenstein index $2$. Its index $1$ cover is $\frac{1}{2}(1,1)$, which is the hypersurface singularity $(xy - z^2 = 0)$ in $\AA^3_{x,y,z} = \Spec \CC[x,y,z]$. Therefore $\frac{1}{4}(1,1)$ is the closed subscheme of the $3$-fold quotient singularity $\frac{1}{2}(1,1,1)_{x,y,z}$
given, with respect to the orbifold coordinates $x,y,z$, by the equation $xy - z^2 = 0$.

Since the miniversal deformation of $\frac{1}{2}(1,1)$ is given by $xy - z^2 + t = 0$ in $\AA^3_{x,y,z}$ over $\CC \pow{t}$,
we have that the miniversal $\QQ$-Gorenstein deformation of $\frac{1}{4}(1,1)$ is given by
\begin{equation} \label{eq:miniversal_deformation_1411}
	xy - z^2 + t = 0
\end{equation}
 inside $\frac{1}{2}(1,1,1)_{x,y,z}$ over $\CC \pow{t}$.
This specifies a formal morphism
\begin{equation} \label{eq:equivalence_qGdef_1411}
\mathrm{Spf}(\CC \pow{t}) \longrightarrow \DefqG \left( \frac{1}{4}(1,1) \right),
\end{equation}
which is smooth and induces an isomorphism on tangent spaces. Here $\mathrm{Spf}$ denotes the formal spectrum of a local noetherian $\CC$-algebra.
We will always use this morphism when considering the $\QQ$-Gorenstein deformation functor of the singularity $\frac{1}{4}(1,1)$.

\medskip

Now we make a calculation which will be useful in \S\ref{sec:linear_system_of_X_in_Y}. Consider the $2$-parameter deformation
\begin{equation} \label{eq:equazione_deformazione_1411}
xy - z^2 + s_1 + s_2 z^4 = 0
\end{equation}
in $\frac{1}{2}(1,1,1)_{x,y,z}$ over $\CC \pow{s_1, s_2}$.
By versality this deformation comes from the miniversal deformation \eqref{eq:miniversal_deformation_1411} via pull-back along a formal morphism
\begin{equation} \label{eq:formal_morphism}
\mathrm{Spf}(\CC \pow{s_1,s_2}) \longrightarrow \mathrm{Spf}(\CC \pow{t}) \overset{\eqref{eq:equivalence_qGdef_1411}}\longrightarrow \DefqG \left( \frac{1}{4}(1,1) \right),
\end{equation}
which is induced by a local $\CC$-algebra homomorphism $\CC \pow{t} \to \CC \pow{s_1, s_2}$.
Via the automorphism of $\frac{1}{2}(1,1,1)_{x,y,z} \times \mathrm{Spf}(\CC \pow{s_1, s_2})$ given by
\begin{equation*}
z \mapsto z \sqrt{1 - s_2 z^2} = \sum_{n = 0}^\infty \frac{(2n)!}{4^n (n!)^2 (1-2n)} s_2^n z^{n+1}
\end{equation*}
we get an isomorphism of the deformation \eqref{eq:equazione_deformazione_1411} with $xy - z^2 + s_1 = 0$, which is exactly the miniversal deformation~\eqref{eq:miniversal_deformation_1411} once we use the equality $t = s_1$. Therefore the morphism in 
\eqref{eq:formal_morphism} is induced by the local $\CC$-algebra homomorphism $\CC \pow{t} \to \CC \pow{s_1, s_2}$ given by $t \mapsto s_1$.

\subsection{The surface $X$} \label{sec:surface_X}
In the lattice $N = \ZZ^2$ consider the polygon $P$ which is the convex hull of the points
\begin{equation*}
\begin{pmatrix}
2 \\ 1
\end{pmatrix}
, \
\begin{pmatrix}
1 \\ 2
\end{pmatrix}
, \
\begin{pmatrix}
-1 \\ 2
\end{pmatrix}
, \
\begin{pmatrix}
-2 \\ -1
\end{pmatrix}
, \
\begin{pmatrix}
-1 \\ -2
\end{pmatrix}
, \
\begin{pmatrix}
1 \\ -2
\end{pmatrix}
\end{equation*}
and is depicted in Figure~\ref{fig:esagono_nero_rosso}.
(The meaning of the red segments in this figure will be clear in \S\ref{sec:scaffolding}.)
It is clear that $P$ is a \emph{Fano polytope}, i.e.\ it is a lattice polytope such that the origin is in the interior and the vertices are primitive lattice points.
Because of this we can consider the face fan (also called spanning fan) of $P$: this is the collection of cones (with apex at the origin) over the faces of $P$; it is made up of $6$ rational cones in $N$.

\begin{figure}
	\includegraphics[width=0.4\textwidth]{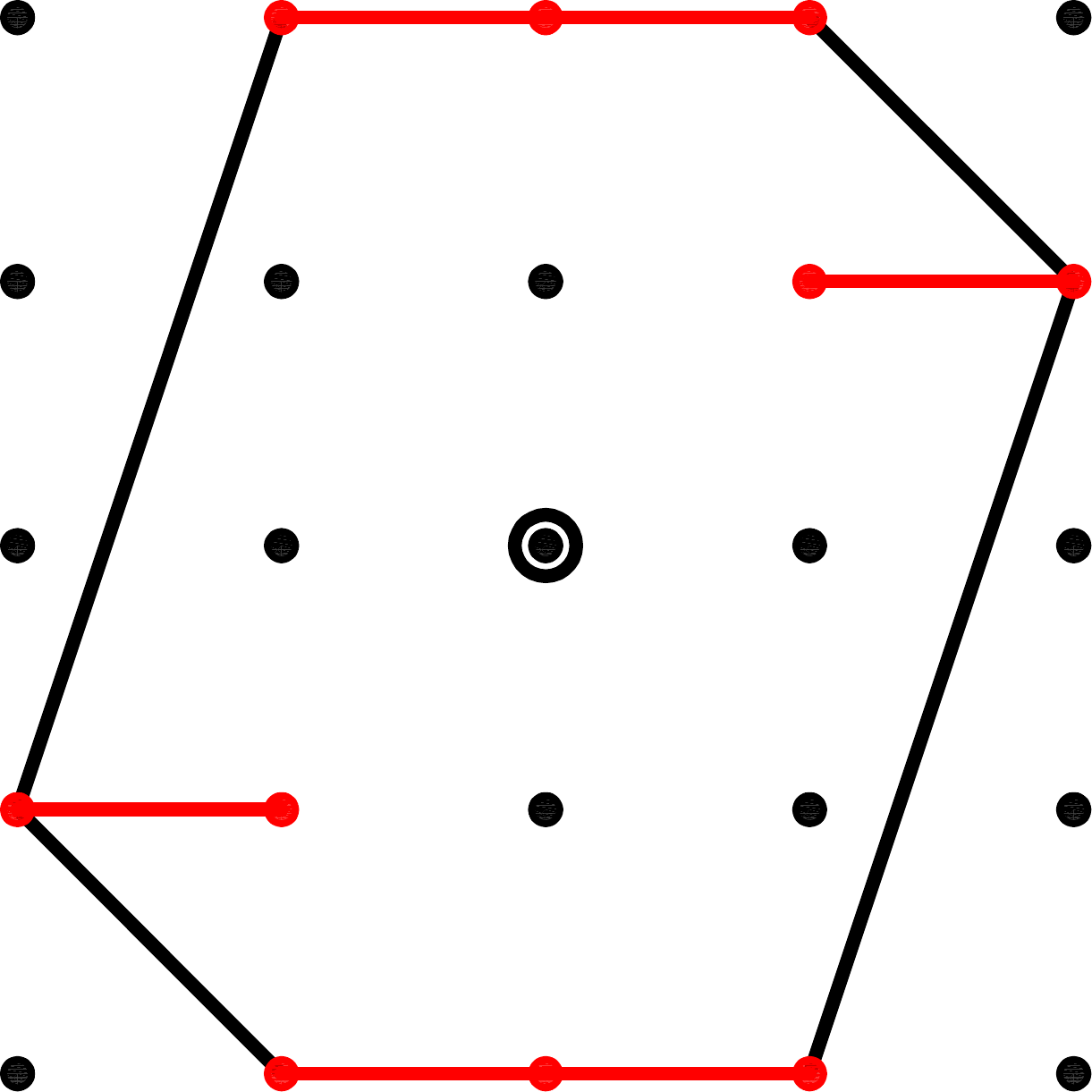}
	\caption{The polygon $P$ in \S\ref{sec:surface_X}}
	\label{fig:esagono_nero_rosso}
\end{figure}

\begin{prop} Let $X$ be the toric variety associated to the face fan of $P$. Then:
\begin{enumerate}
	\item $X$ is a K-polystable toric del Pezzo surface with anticanonical volume $\frac{22}{15}$;
	\item the surface $X$ has exactly $6$ singular points: $2$ points of type $\frac{1}{3}(1,1)$, $2$ points of type $\frac{1}{4}(1,1)$, $2$ points of type $\frac{1}{5}(1,2)$;
	\item the automorphism group $\Aut(X)$ is isomorphic to $(\CC^*)^2 \rtimes C_2$, where $C_2$ is the cyclic group of order $2$ and the non-trivial element of $C_2$ acts on $(\CC^*)^2$ via $(z,w) \mapsto (z^{-1}, w^{-1})$.
\end{enumerate}
\end{prop}

\begin{proof}
	(1) Since $N$ has rank $2$, the dimension of $X$ is $2$. By a slight modification of \cite[Theorem~8.3.4]{cox}, since the fan of $X$ is the face fan of a Fano polytope, we have that $X$ is Fano.
	
	Let $P^\circ$ denote the polar of $P$ (see \cite[\S2.4]{ask_petracci}); it is a rational polytope in the dual lattice $M = \Hom_\ZZ(N,\ZZ)$ and is the moment polytope of the toric boundary of $X$, which is the reduced sum of the torus invariant prime divisors of $X$ and is an anticanonical divisor.
	The anticanonical volume of $X$ is the normalised volume of $P^\circ$, which is $\frac{22}{15}$.
	Here the normalised volume is the double of the Lebesgue measure: in this way the normalised volume of a unimodular simplex is $1$.
	Since $P$ is centrally symmetric (i.e.\ $P = -P$), also $P^\circ$ is centrally symmetric, hence the barycentre of $P^\circ$ is the origin. Therefore $X$ is K-polystable by \cite{berman_polystability}.
	
	In order to prove (2) one needs to analyse the six $2$-dimensional cones of the face fan of $P$ and apply \cite[\S10.1]{cox}. For instance, the two horizontal edges of $P$ give the two $\frac{1}{4}(1,1)$ singularities.
	
	(3) let $T_N$ denote the $2$-dimensional torus $N \otimes_\ZZ (\CC^*)^2 = \Spec \CC[M]$ which acts on $X$.
	Let $\Aut(P)$ be the group of the symmetries of $P$: it is generated by $-\mathrm{id}_N$.
	Since every facet of $P^\circ$ has no interior lattice points, by \cite[Proposition~2.8]{ask_petracci}  $\Aut(X)$ is the semidirect product $T_N \rtimes \{ \pm \mathrm{id}_N \}$.	
	\end{proof}

The points of type $\frac{1}{3}(1,1)$ and $\frac{1}{5}(1,2)$ are $\QQ$-Gorenstein rigid, i.e.\ they do not deform $\QQ$-Gorensteinly.
The $\QQ$-Gorenstein deformations of $\frac{1}{4}(1,1)$ have been considered in \S\ref{sec:deformations_1/4(1,1)}.

By \cite[Lemma~6]{procams} there are no local-to-global obstructions for $\QQ$-Gorenstein deformations of $X$, so the $\QQ$-Gorenstein smoothings of the two $\frac{1}{4}(1,1)$ points of $X$, which we denote $p_1$ and $p_2$, can be realised globally and simultaneously.
More precisely, since $\rH^i(X,T_X) = 0$ for $i \geq 1$ by \cite{petracci_survey}, the product of the restriction morphisms to the germs $(p_i \in X)$
\begin{equation} \label{eq:restriction_deformation_functor}
\DefqG(X) \longrightarrow \DefqG(p_1 \in X) \times \DefqG(p_2 \in X)
\end{equation}
is smooth and induces an isomorphism on tangent spaces.
So $\CC \pow{t_1, t_2}$ is the hull of $\DefqG(X)$
and $t_i$ is the $\QQ$-Gorenstein smoothing parameter of $(p_i \in X)$.
Here the parameter $t_i$ is defined through \eqref{eq:equivalence_qGdef_1411}.
In the next section we will realise the miniversal $\QQ$-Gorenstein deformation of $X$ in a linear system in a toric Fano $3$-fold.

\begin{prop} \label{prop:M_and_cM}
Let $\cM$ (resp.\ $M$) be the connected component of the K-moduli stack $\stack{2}{\frac{22}{15}}$ (resp.\ the K-moduli space $\modspace{2}{\frac{22}{15}}$) which contains the point corresponding to $X$.
Then $M$ is a smooth projective irreducible curve.
\end{prop}

\begin{proof}
	Since $\QQ$-Gorenstein deformations of del Pezzo surfaces  are unobstructed by \cite[Lemma~6]{procams}, by \cite[Remark~2.4]{ask_petracci} we get that $\cM$ is smooth and $M$ is normal. Moreover $M$ is projective by \cite{projectivity_K_moduli_final}.

The automorphism group $\Aut(X)$ acts on the hull $\CC \pow{t_1, t_2}$.
The weights of $t_1$ (resp.\ $t_2$) in $M$ is $(0,1)$ (resp.\ $(0,-1)$).
Therefore the invariant subring of the formal action of $T_N = (\CC^*)^2$ on $\CC \pow{t_1, t_2}$ is $\CC \pow{t_1 t_2}$.
The group $C_2$ swaps $t_1$ and $t_2$, so it leaves $t_1 t_2$ invariant.
Therefore the invariant subring of the formal action of $\Aut(X)$ on $\CC \pow{t_1, t_2}$ is $\CC \pow{t_1 t_2}$.

By the Luna  \'etale slice theorem for algebraic stacks \cite{luna_etale_slice_stacks} the local structure of $\cM \to M$ is given by the  commutative square
\begin{equation*}
	\xymatrix{
		\left[ \Spf \CC \pow{t_1, t_2} \ / \ \Aut(X)  \right] \ar[d] \ar[r] & \cM \ar[d] \\
		\Spf \CC \pow{t_1 t_2} \ar[r] & M
	}
\end{equation*}
where the horizontal maps are formally \'etale and maps the closed point to $[X]$.
This implies that $M$ has dimension $1$.
Hence  $M$ is a smooth projective curve.
\end{proof}

\subsection{The $3$-fold $Y$ and the proof of Theorem~\ref{thm:main}}
\label{sec:conclusion_proof}

Consider $\AA^6$ with coordinates $x_1, x_2, y_1, y_2, z_1, z_2$.
Consider the toric $3$-fold $Y$ given by the GIT quotient $\AA^6 / \! \! / (\CC^*)^3$ where the linear action of $(\CC^*)^3$ on $\AA^6$ is specified by the weights
\begin{equation*}
	\begin{array}{cccccc|c}
		x_1 & x_2 & y_1 & y_2 & z_1 & z_2 &  \\
		\hline
		0 & 0 & 1 & 1 & 1 & 1 & L_1 \\
		0 & 1 & 3 & 1 & 0 & 6 & L_2 \\
		1 & 0 & 1 & 3 & 6 & 0 & L_3
	\end{array}
\end{equation*}
and by the stability condition whose unstable locus is the vanishing locus of the ideal 
\begin{equation} \label{eq:irrelevant_ideal}
	(x_1,x_2,  z_1) \cdot ( x_1, x_2, z_2) \cdot (y_1, y_2) \cdot (y_1, z_2) \cdot (y_2, z_1)
\end{equation}
in the polynomial ring $\CC[x_1, x_2, y_1, y_2, z_1, z_2]$.
Now $L_1, L_2, L_3$ are the $\QQ$-line bundles on $Y$ which come from the standard basis of the character lattice of $(\CC^*)^3$. They form a $\ZZ$-basis of the divisor class group of $Y$.

We see that $\rH^0(Y,  2 L_1 + 6 L_2 + 6 L_3 )$ has dimension $4$ and its monomial basis is made up of the  monomials
\begin{equation*}
	z_1 z_2, \
	y_1 y_2 x_1^2 x_2^2, \
	x_1^4 y_1^2, \
	x_2^4 y_2^2.
\end{equation*}
We consider a special affine subspace $\rH^0(Y,  2 L_1 + 6 L_2 + 6 L_3 )$ and we relate to the surface $X$ considered in \S\ref{sec:surface_X}:

\begin{prop} \label{prop:X_inside_Y} Let $Y$ be the toric $3$-fold defined above. Let $X$ be the toric del Pezzo surface considered in \S\ref{sec:surface_X}. Consider the flat family  $\cX \to \AA^2 = \Spec \CC [ s_1, s_2]$ of hypersurfaces in the linear system $\vert 2 L_1 + 6 L_2 + 6 L_3 \vert$ on $Y$ defined by the equation
\begin{equation} \label{eq:perturbed_hypersurface}
	z_1 z_2 - y_1 y_2 x_1^2 x_2^2 + s_1 x_1^4 y_1^2 + s_2 x_2^4 y_2^2 = 0.
\end{equation}
Then:
\begin{enumerate}
	\item[(A)] the fibre of $\cX \to \AA^2$ over the origin $0 \in \AA^2$ is the toric surface $X$;
	\item[(B)]  the base change of $\cX \to \AA^2$ to $\CC \pow{s_1, s_2}$ is the miniversal $\QQ$-Gorenstein deformation of $X$.
\end{enumerate}
\end{prop}

We postpone the proof of this proposition: the proof of (A) is given in \S\ref{sec:scaffolding} and the proof of (B) is given in \S\ref{sec:linear_system_of_X_in_Y}. Now we show how this proposition implies our main result.

\begin{proof}[Proof of Theorem~\ref{thm:main}]
	Let $\cM$ and $M$ be as in Proposition~\ref{prop:M_and_cM}. We have that $M$ is a smooth projective irreducible curve. We want to show that $M$ is isomorphic to $\PP^1$.

Let $\cX \to \AA^2$ be the $\QQ$-Gorenstein family considered in Proposition~\ref{prop:X_inside_Y}. Since the central fibre is K-polystable, by openness of K-semistability \cite{BLX}, there exists an open neighbourhood $U$ of the origin in $\AA^2$ such that the fibred product $\cX \times_{\AA^2} U \to U$ induces a morphism $U \to \cM$, which is formally smooth at the origin.

By looking at the action of $\mathrm{Aut}(X)$ on the base of the miniversal $\QQ$-Gorenstein deformation of $X$ (see the proof of Proposition~\ref{prop:M_and_cM}), we see that there are K-polystable surfaces in $U$ non-isomorphic to $X$. Therefore, by composing $U \to \cM$ with $\cM \to M$, we get a non-constant morphism $U \to M$.
By restricting to a general line passing through the origin in $U \subseteq \AA^2$, we get that $M$ is unirational. Therefore $M$ is rational by L\"uroth's theorem.
This concludes the proof of Theorem~\ref{thm:main}.
\end{proof}

\subsection{Proof of Proposition~\ref{prop:X_inside_Y}(A)}
\label{sec:scaffolding}

We need to prove that the surface $X$ is the hypersurface in the $3$-fold $Y$ defined by the equation $z_1 z_2 - y_1 y_2 x_1^2 x_2^2 = 0$.
We apply the Laurent inversion method \cite{laurent_inversion, from_cracked_to_fano, cracked_fano_toric_ci}.

Let $e_1, e_2$ be the standard basis of $N = \ZZ^2$. Consider the decomposition
\[
N = \overline{N} \oplus N_U
\]
where $\overline{N} = \ZZ e_1$ and $N_U = \ZZ e_2$.
Let $\overline{M}$ be the dual lattice of $\overline{N}$.
Let $Z$ be the $T_{\overline{M}}$-toric variety associated to complete fan in the lattice $\overline{M}$ with rays generated by $e_1^*$ and $-e_1^*$.
It is clear that $Z$ is isomorphic to $\PP^1$.
Let $\mathrm{Div}_{T_{\overline{M}}}(Z)$ be the rank-$2$ lattice consisting of the torus invariant divisors on $Z$:
a basis of $\mathrm{Div}_{T_{\overline{M}}}(Z)$ is given by the torus invariant prime divisors on $Z$, namely $E_+$, $E_-$, which are associated to the rays $e_1^*$, $-e_1^*$ respectively.
The divisor sequence \cite[Theorem~4.1.3]{cox} of $Z$ is
\[
0 \longrightarrow \overline{N} = \ZZ e_1
\xrightarrow{
	\rho^\star =
\begin{pmatrix}
1 \\ -1
\end{pmatrix}
}
\mathrm{Div}_{T_{\overline{M}}}(Z) = \ZZ E_+ \oplus \ZZ E_-
\xrightarrow{
\begin{pmatrix}
1 & 1
\end{pmatrix}
}
\Pic(Z) =\ZZ \longrightarrow 0.
\]
We consider the following ample torus invariant divisors on $Z$
\begin{align*}
D_{x_1} &= E_+ + E_- &  D_{y_1} &= - E_+ + 2 E_- \\
D_{x_2} &= E_+ + E_- &  D_{y_2} &= 2E_+ - E_-
\end{align*}
and their corresponding moment polytopes in $\overline{N}$:
\begin{align*}
P_{D_{x_1}} &= \conv{-e_1, e_1} &  P_{D_{y_1}} &= \conv{e_1, 2 e_1}, \\
P_{D_{x_2}} &= \conv{-e_1, e_1} &   P_{D_{y_2}} &= \conv{-2 e_1, -e_1}.
\end{align*}
Now consider the following elements in the lattice $N_U = \ZZ e_2$:
\begin{align*}
	\chi_{x_1} &= 2e_2 & 	 		\chi_{y_1} &= e_2 \\
	 	\chi_{x_2} &= -2e_2 &  	 			\chi_{y_2} &=  -e_2.
\end{align*}
The polytopes 
\begin{align*}
P_{D_{x_1}} + \chi_{x_1} &\qquad P_{D_{y_1}} + \chi_{y_1} \\
P_{D_{x_2}} + \chi_{x_2} &\qquad P_{D_{y_2}} + \chi_{y_2}
\end{align*}
 in $N = \overline{N} \oplus N_U$ are the four red segments in Figure~\ref{fig:esagono_nero_rosso}.
Clearly the polygon $P$ is the convex hull of these four segments.
By \cite[Definition~3.1]{laurent_inversion} the set
\begin{equation*}
S = \left\{
\left( D_{x_1}, \chi_{x_1}   \right), \ 
\left( D_{x_2}, \chi_{x_2}   \right), \ 
\left( D_{y_1}, \chi_{y_1}   \right), \ 
\left( D_{y_2}, \chi_{y_2}   \right)
   \right\}
\end{equation*}
is a `scaffolding' on the Fano polygon $P$.

Consider the rank-$3$ lattice $\tilde{N} := \mathrm{Div}_{T_{\overline{M}}}(Z) \oplus N_U = \ZZ E_+ \oplus \ZZ E_- \oplus \ZZ e_2$.
Let $\tilde{M}$ be the dual lattice of $\tilde{N}$ and let $\langle \cdot, \cdot \rangle \colon \tilde{M} \times \tilde{N} \to \ZZ$ be the duality pairing.
Following \cite[Definition~A.1]{laurent_inversion} we consider the polytope $Q_S \subseteq \tilde{M}_\RR$ defined by the following inequalities:
\begin{align*}
\langle \ \cdot \ , - D_{x_1} + \chi_{x_1} \rangle &\geq -1, \\
\langle \ \cdot \ , - D_{x_2} + \chi_{x_2} \rangle &\geq -1, \\
\langle \ \cdot \ , - D_{y_1} + \chi_{y_1} \rangle &\geq -1, \\
\langle \ \cdot \ , - D_{y_2} + \chi_{y_2} \rangle &\geq -1, \\
\langle \ \cdot \ , E_+ \rangle &\geq 0, \\
\langle \ \cdot \ , E_- \rangle &\geq 0.
\end{align*}
Let $\Sigma_S$ be the normal fan of $Q_S$. One can see that $\Sigma_S$ is the complete simplicial fan in $\tilde{N} = \mathrm{Div}_{T_{\overline{M}}}(Z) \oplus N_U$ with rays generated by the following vectors:
\begin{align*}
x_1 &= - D_{x_1} + \chi_{x_1} = - E_+ - E_- + 2 e_2\\
x_2 &= - D_{x_2} + \chi_{x_2} = - E_+ - E_- - 2 e_2 \\
y_1 &= - D_{y_1} + \chi_{y_1} = E_+ - 2 E_- + e_2 \\
y_2 &= - D_{y_2} + \chi_{y_2} = - 2 E_+ + E_- - e_2\\
z_1 &= E_+, \\
z_2 &= E_-.
\end{align*}

Let $Y$ be the $T_{\tilde{N}}$-toric variety associated to the fan $\Sigma_S$.
Thus $Y$ is a $\QQ$-factorial Fano $3$-fold with Cox coordinates $x_1, x_2, y_1, y_2, z_1, z_2$ .
With respect to the basis of $\tilde{N}$ given by $E_+$, $E_-$, $e_2$, the rays of the fan $\Sigma_S$ are the columns of  the matrix
\[
\begin{pmatrix}
-1 & -1 & 1 & -2 & 1 & 0 \\
-1 & -1 & -2 & 1 & 0 & 1 \\
2 & -2 & 1 & -1 & 0 & 0
\end{pmatrix}.
\]
The transpose of this matrix gives an injective $\ZZ$-linear homomorphism $\tilde{M} \to \ZZ^6$. By \cite[Theorem~4.1.3]{cox} the cokernel of this is the divisor map of $Y$ and is isomorphic to the divisor class group of $Y$.
In this case, one finds that the divisor map  of $Y$ is the $\ZZ$-linear homomorphism $\ZZ^6 \to \mathrm{Cl}(Y) \simeq \ZZ^3$ given by the following matrix.
\begin{equation*}
\begin{array}{cccccc|c}
	x_1 & x_2 & y_1 & y_2 & z_1 & z_2 &  \\
	\hline
0 & 0 & 1 & 1 & 1 & 1 & L_1 \\
0 & 1 & 3 & 1 & 0 & 6 & L_2 \\
1 & 0 & 1 & 3 & 6 & 0 & L_3
\end{array}
\end{equation*}
Here $L_1, L_2, L_3$ are the elements of the chosen $\ZZ$-basis of $\mathrm{Cl}(Y)$.
This $3 \times 6$ matrix gives the weights of a linear action of the torus $(\CC^*)^3$ on $\AA^6$.
By \cite[\S5.1]{cox}
$Y$ is the GIT quotient of this action with respect to the stability condition given by the irrelevant ideal
\[
(x_1,x_2,  z_1) \cdot ( x_1, x_2, z_2) \cdot (y_1, y_2) \cdot (y_1, z_2) \cdot (y_2, z_1).
\]
Therefore $Y$ is the toric $3$-fold considered in \S\ref{sec:conclusion_proof}.

We now consider the injective linear map
\[
\theta := \rho^\star \oplus \mathrm{id}_{N_U} \colon
N = \overline{N} \oplus N_U \longrightarrow
\tilde{N} = \mathrm{Div}_{T_{\overline{M}}}(Z) \oplus N_U.
\]
By \cite[Theorem~5.5]{laurent_inversion} $\theta$ induces a toric morphism $X \to Y$ which is a closed embedding.
We want to understand the ideal of this closed embedding in the Cox ring of $Y$ by using the map $\theta$.

We follow \cite[Remark~2.6]{from_cracked_to_fano}.
We see that $\theta(N)$ is the hyperplane defined by the vanishing of $h = E_+^* + E_-^* \in \tilde{M}$.
Now we compute the duality pairing between $h$ and the primitive generators of the rays of $\Sigma_S$:
$\langle h, x_1 \rangle = \langle h, x_2 \rangle = -2$,
$\langle h, y_1 \rangle = \langle h, y_2 \rangle = -1$,
$\langle h, z_1 \rangle = \langle h, z_2 \rangle = 1$.
We get that the polynomial
\begin{equation} \label{eq:hypersurface}
z_1 z_2 - y_1 y_2 x_1^2 x_2^2
\end{equation}  is the generator of the ideal of the closed embedding $X \into Y$ in the Cox ring of $Y$.
In other words, $X$ is the hypersurface in $Y$ defined by the vanishing of the polynomial~\eqref{eq:hypersurface} in the Cox coordinates of $Y$.
This concludes of (A) in Proposition~\ref{prop:X_inside_Y}.

\subsection{Proof of Proposition~\ref{prop:X_inside_Y}(B)}
\label{sec:linear_system_of_X_in_Y}

We want to show that, after base change to $\CC \pow{s_1, s_2}$, the family of hypersurfaces in $Y$ defined by the vanishing of \eqref{eq:perturbed_hypersurface}
is the miniversal $\QQ$-Gorenstein deformation of $X$.
Since the map in \eqref{eq:restriction_deformation_functor} is smooth and induces an isomorphism on tangent spaces, we need to check that locally this family induces the miniversal deformations of the singularity germs of $X$.
Let $t_1$ and $t_2$ be the two smoothing parameters of the two $\frac{1}{4}(1,1)$ singularities of $X$, as fixed in \S\ref{sec:deformations_1/4(1,1)}.
We proceed by analysing each chart of the affine open cover of $Y$ given by the fan $\Sigma_S$.
\begin{itemize}
	
					\item The cone $\sigma_{ x_1, z_1, z_2}$ gives the isolated singularity $\frac{1}{2}(1,1,1)_{ x_1, z_1, z_2}$ on $Y$.
					In this chart, by dehomogenising \eqref{eq:perturbed_hypersurface}, we get the equation
	$z_1 z_2 - x_1^2  + s_1 x_1^4  + s_2 = 0$ in the orbifold coordinates. This is exactly the $\QQ$-Gorenstein smoothing of $\frac{1}{4}(1,1)$ described at the end of \S\ref{sec:deformations_1/4(1,1)}. So we have $t_2 = s_2$.
	
	\item The cone $\sigma_{ x_2, z_1, z_2}$ gives the isolated singularity $\frac{1}{2}(1,1,1)_{ x_2, z_1, z_2}$ on $Y$.
	In this chart we get the equation
$z_1 z_2 -   x_2^2 + s_1  + s_2 x_2^4 = 0$.
We are in a completely analogous situation as the previous case, so  $t_1 = s_1$.
	
\item The cone $\sigma_{ x_1, y_2, z_2}$ gives the isolated singularity $\frac{1}{5}(2,1,4)_{ x_1, y_2, z_2}$ on $Y$.
In this chart we get the equation
$z_2 - y_2 x_1^2  + s_1 x_1^4  + s_2  y_2^2 = 0$, which is quasi-smooth because there is no constant term and $z_2$ appears with degree $1$. So all fibres of $\cX \to \AA^2$ have a $\frac{1}{5}(1,2)$ singularity at the $0$-stratum of this chart of $Y$.

\item The cone $\sigma_{ x_2, y_1, z_1}$ gives the isolated singularity $\frac{1}{5}(1,2,4)_{ x_2, y_1, z_1}$ on $Y$. The equation is
$ z_1 - y_1   x_2^2 + s_1  y_1^2 + s_2 x_2^4 = 0$ and, in a way analogous to the previous case, we get a $\frac{1}{5}(1,2)$ singularity on every fibre of $\cX \to \AA^2$ at the $0$-stratum of this chart of $Y$.

			\item The cone $\sigma_{ x_1, y_1, z_1}$ gives the non-isolated singularity $\frac{1}{3}(1,1,0)_{ x_1, y_1, z_1}$.
			The equation is
$z_1 - y_1 x_1^2  + s_1 x_1^4 y_1^2 + s_2 = 0$. Since it is quasi-smooth, this gives a $\frac{1}{3}(1,1)$ singularity on every fibre of $\cX \to \AA^2$ at a point on the curve $(x_1 = y_1 = 0) \subset Y$.
	
	\item The cone $\sigma_{ x_2, y_2, z_2}$ gives the non-isolated singularity $\frac{1}{3}(1,1,0)_{ x_2, y_2, z_2}$. The equation is
$ z_2 - y_2  x_2^2 + s_1  + s_2 x_2^4 y_2^2 = 0$ and, similarly to the previous case, we have a $\frac{1}{3}(1,1)$ singularity on every fibre of $\cX \to \AA^2$ at a point on the curve $(x_2 = y_2 = 0) \subset Y$.
	
		\item In the fan $\Sigma_S$ there are two $3$-dimensional cones which we have not been analysed yet: these are $\sigma_{x_1, x_2, y_1}$, whose corresponding chart on $Y$ is the non-isolated singularity $\frac{1}{12}(3,1,4)_{x_1, x_2, y_1}$, and $\sigma_{x_1, x_2, y_2}$, which gives the non-isolated singularity $\frac{1}{12}(4,1,3)_{x_1, x_2, y_2}$ on $Y$.
		We want to show that it is useless to analyse these cones.
		Let $V$ denote
		the complement in $Y$ of the union of the already analysed charts; $V$ is made up of $3$ torus-orbits: the $0$-stratum corresponding to $\sigma_{x_1, x_2, y_1}$, the $0$-stratum corresponding to $\sigma_{x_1, x_2, y_2}$, and the $1$-stratum corresponding to $\sigma_{x_1, x_2}$.
		In other words $V$ is the projective curve $(x_1 = x_2 = 0)$ in $Y$. By looking at the equation \eqref{eq:perturbed_hypersurface} and at the irrelevant ideal~\eqref{eq:irrelevant_ideal} it is clear that $V$ does not intersect any fibre of $\cX \to \AA^2$.
		
\end{itemize}

To sum up, we have that the family $\cX \to \AA^2$ realises the $\QQ$-Gorenstein smoothings of the two $\frac{1}{4}(1,1)$ points on $X$ and leaves the $\frac{1}{3}(1,1)$ points and $\frac{1}{5}(1,2)$ points untouched (i.e.\ the deformation is formally isomorphic to a product around these points of the central fibre).
By versality the family $\cX \to \AA^2$ induces a morphism $\mathrm{Spf}(\CC \pow{s_1, s_2}) \to \DefqG(X)$, which is associated to the isomorphism $\CC \pow{s_1, s_2} \simeq \CC \pow{t_1, t_2}$, where $s_1 = t_1$ and $s_2 = t_2$.
In other words, the base change of $\cX \to \AA^2$ to $\CC \pow{s_1, s_2}$ is the miniversal $\QQ$-Gorenstein deformation of $X$.
This concludes the proof of Proposition~\ref{prop:X_inside_Y}(B).


\section{Mirror symmetry}
\label{sec:mirror_symmetry}

In \cite{procams} some conjectures for del Pezzo surfaces were formulated.
In this section we sketch some evidence for these conjectures in the case of the toric del Pezzo surface $X$ and of its $\QQ$-Gorenstein deformations.
In addition to \cite{procams}, we refer the reader to \cite{fanosearch, petracci_roma, quantumFano3folds} and to the references therein for more details about the notions introduced below.

\subsection{Combinatorial avatars of connected components of moduli of del Pezzo surfaces}
According to \cite[Conjecture~A]{procams} 
there is a $1$-to-$1$ correspondence between
\begin{itemize}
	\item connected components of the moduli stack of del Pezzo surfaces (with a toric degeneration) 
	and
	\item    mutation equivalence classes of Fano polygons.
\end{itemize}
Here a \emph{Fano polygon} is a lattice polygon whose face fan defines a del Pezzo surface (an example is $P$ in \S\ref{sec:surface_X});
and \emph{mutation} is a certain equivalence relation on Fano polygons introduced in \cite{sigma} --- we do not give further details here and we refer the reader to \cite{procams, minimal_polygons}.

The correspondence works in the following way: to (the mutation equivalence class of) the Fano polygon $P$ one associates the connected component $\scrM$ of the moduli stack of  del Pezzo surfaces which contains the surface $X$, which is the toric del Pezzo surface associated to the face fan of $P$.
One has that $\scrM$ is smooth and contains $\cM$ (the connected component of the K-moduli stack parametrising K-semistable del Pezzo surfaces and containing $X$) as an open substack.

\subsection{Classical period} \label{sec:classical_period}

Consider the family of \emph{maximally mutable} Laurent polynomials with Newton polytope $P$ and with \emph{T-binomial edge coefficients} \cite[Definition~4]{procams}.
This is the $6$-dimensional family
\begin{align*}
	f &= x^2 y + x^{-2} y^{-1} + (x + 2 + x^{-1}) (y^2 + y^{-2}) \\
	&+	a_1 xy + a_2 x^{-1} y^{-1} + b_1 x + b_2 x^{-1} + c_1 x y^{-1} + c_2 x^{-1} y
\end{align*}
in $\QQ[a_1, a_2, b_1, b_2, c_1, c_2][x^\pm, y^\pm]$, where $a_1, a_2, b_1, b_2, c_1, c_2$ are indeterminates.
\begin{figure}
	\centering
	\def\svgwidth{0.4\textwidth}
	\input{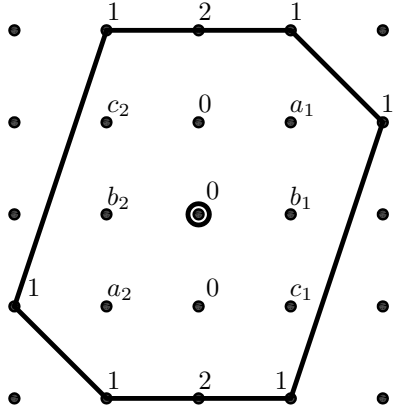}
	\bigskip
	\caption{The coefficients of the maximally mutable Laurent polynomials with Newton polytope $P$ and with T-binomial edge coefficients (see \S\ref{sec:classical_period})}
	\label{fig:polinomio}
\end{figure}
In Figure~\ref{fig:polinomio} the coefficients of $f$ are written next to the corresponding lattice points of $P$.

The \emph{classical period} of $f$ is the power series
\begin{equation*}
\pi_f(t) = \left( \frac{1}{2 \pi \mathrm{i}} \right)^2 \int_{\{ (x,y) \in (\CC^*)^2 \mid |x| = |y| = \epsi \}} \frac{1}{1-t f(x,y)} \frac{\rmd x}{x} \frac{\rmd y}{y}
\end{equation*}
in $\QQ[a_1, a_2, b_1, b_2, c_1, c_2] \pow{t}$, for some $0 < \epsi \ll 1$.
The first coefficients of $\pi_f$ are:
\begin{align*}
\pi_f(t) &= 1 +
2( a_1 a_2 + b_1 b_2 + c_1 c_2 + 7) t^2 + \\
&+ 6 (a_1 b_1 + 2a_1 c_2 + a_2 b_2 + 2 a_2 c_1 + 4 b_1 + 4 b_2 + c_1 + 
c_2) t^3 + \cdots.
\end{align*}

\subsection{Quantum period} \label{sec:quantum}

Let $\Xprime$ be the surface corresponding to a general point in $\scrM$; in other words, $\Xprime$  is a general $\QQ$-Gorenstein deformation of the toric surface $X$.
The \emph{quantum period} of $\Xprime$ \cite[Definition~3.2]{oneto_petracci} is a certain generating function for genus zero Gromov--Witten invariants of $\Xprime$ which depends on certain parameters related to the singularities of $\Xprime$.
In this case there are $6$ parameters because the singular locus of $\Xprime$ is made up of $2$ points of type $\frac{1}{3}(1,1)$ and $2$ points of type $\frac{1}{5}(1,2)$.

In general it is very difficult to compute the quantum period of a Fano orbifold.
Since $\Xprime$ is a hypersurface in the toric Fano $Y$, one can use the quantum Lefschetz theorem to compute a specialisation of the quantum period of $\Xprime$, i.e.\ the power series $G_\Xprime \in \QQ \pow{t}$ obtained from the quantum period by setting the parameters equal to some numbers. This can be done as follows.
We use the notation as in \S\ref{sec:scaffolding}. One can see that the nef cone of $Y$ is spanned by the divisor classes
\begin{gather*}
	L_1 + 3 L_2 + 3 L_3, \\
	4 L_1 + 9 L_2 + 9 L_3, \\
	5 L_1 + 9 L_2 + 15 L_3, \\
	5L_1 + 15 L_2 + 9 L_3.
	\end{gather*}
We consider the cone $\Lambda \subseteq \RR^3$ defined by the inequalities
\begin{gather*}
	l_1 + 3 l_2 + 3 l_3 \geq 0, \\
	4 l_1 + 9 l_2 + 9 l_3 \geq 0 , \\
	5 l_1 + 9 l_2 + 15 l_3 \geq 0, \\
	5l_1 + 15 l_2 + 9 l_3 \geq 0
\end{gather*}
and by the inequalities
\begin{gather*}
	l_3 \geq 0, \\
	l_2 \geq 0, \\
	l_1 + 3 l_2 + l_3 \geq 0, \\
	l_1 + l_2 + 3 l_3 \geq 0, \\
	l_1 + 6 l_3 \geq 0, \\
	l_1 + 6 l_2 \geq 0.
	\end{gather*}
The first inequalities say that we are taking (the closure of) the cone of the effective curves in $\mathrm{N}_1(Y)_\RR$, i.e.\ we are taking the dual of the nef cone of $Y$;
with the second inequalities we are taking the curve classes  on which the prime torus-invariant divisors of $Y$ have non negative degrees.

By using methods similar to \cite{oneto_petracci}, one can prove that a specialisation of the quantum period of $\Xprime$ is the power series $G_\Xprime(t) \in \QQ \pow{t}$ equal to
\begin{equation*}
\sum_{(l_1, l_2, l_3) \in \Lambda \cap \ZZ^3}
\frac{(2 l_1 +6 l_2 +6 l_3)!}{l_3! \ l_2! \ (l_1 + 3 l_2 + l_3)! \ (l_1 + l_2 + 3 l_3)! \ (l_1 + 6 l_3)! \ (l_1 + 6 l_2)!} t^{2 l_1 +5 l_2 + 5 l_3}.
\end{equation*}
Notice the following numerology: at the denominator there are the factorial of the degrees of the prime torus-invariant divisors of $Y$,
 the numerator is the factorial of the degree of the $\QQ$-line bundle $\cO_Y(\Xprime) = 2L_1 + 6L_2 + 6L_3$,
the exponent of $t$ is the degree of the $\QQ$-line bundle $-K_Y - \Xprime = 2L_1 + 5L_2 + 5 L_3$, which by adjunction restricts to $-K_\Xprime$ on $\Xprime$.

If $\sum_{d \geq 0} C_d t^d$ is the quantum period of $\Xprime$, then the \emph{regularised quantum period} of $\Xprime$ is $\sum_{d \geq 0} d! \ C_d t^d$.
From the computation above one computes the first coefficients of a specialisation of the regularised quantum period of $\Xprime$:
\begin{align*}
	\widehat{G}_\Xprime(t) &= 1 + 16 t^2 + 936 t^4
	+ 520 t^5 + 76840 t^6 + 131880 t^7 + 7360920 t^8 + \\
	&+ 22806000 t^9 + 770459256 t^{10}
	+ 3451657440 t^{11}
	+ 85553394696 t^{12} + \cdots. \\
	\end{align*}

\subsection{Equality of periods}

A second mirror-symmetric expectation \cite[Conjecture~B]{procams}  is that there is an equality between
\begin{itemize}
	\item  the regularised quantum period of a general surface $\Xprime$  in $\scrM$ and
	\item the classical period of the family of maximally mutable Laurent polynomials with Newton polytope $P$ and with T-binomial edge coefficients.
\end{itemize}
Notice that in our case both periods depend on $6$ parameters which should be identified.

Combining \S\ref{sec:classical_period} and \S\ref{sec:quantum} one can verify the equality between a specialisation of the regularised quantum period of $\Xprime$ and the classical period of the Laurent polynomial obtained from $f$ by setting $a_1 = a_2 = 1$ and $b_1 = b_2 = c_1 = c_2 = 0$:
\[
\widehat{G}_\Xprime(t) = \pi_f(t) \vert_{a_1 = a_2 = 1,~b_1 = b_2 = c_1 = c_2 = 0}.
\]


\bibliography{Biblio_note3}
\end{document}